\documentclass{amsart}
\usepackage{graphicx}
\usepackage[mathscr]{eucal}
\usepackage{amssymb}
\newtheorem{thm}{Theorem}
\newtheorem{cor}[thm]{Corollary}
\newtheorem{lem}[thm]{Lemma}
\newtheorem{prop}[thm]{Proposition}
\theoremstyle{definition}

\theoremstyle{remark}
\newtheorem{rem}[thm]{Remark}
\newcommand{\F}{\mathcal{F}}
\newcommand{\Pm}{\mathrm{Prim}}
\newcommand{\mset}{\emptyset}
\newcommand{\Fa}{\mathrm{Fac}}
\newcommand{\Pme}{\mathrm{Prime}}
\begin{document}
\baselineskip=18pt
\title{The space of ideals in the minimal tensor product of $C^*$-algebras}
\author{Aldo J. Lazar}
\address{School of Mathematical Sciences\\
         Tel Aviv University\\
         Tel Aviv 69778, Israel}
\email{aldo@post.tau.ac.il}

\thanks{}%
\date{March 5, 2009}%
%\commby{}%
% ---------------------------------------------------------------
\begin{abstract}

   For $C^*$-algebras $A_1, A_2$ the map $(I_1,I_2)\to \mathrm{ker}(q_{I_1}\otimes q_{I_2})$ from
   $\mathrm{Id}^{\prime}(A_1)\times \mathrm{Id}^{\prime}(A_2)$ into $\mathrm{Id}^{\prime}(A_1\otimes_{\mathrm{min}} A_2)$
   is a homeomorphism
   onto its image which is dense in the range. Here, for a $C^*$-algebra $A$, the space of all proper closed two sided ideals
   endowed with an adequate topology is denoted $\mathrm{Id}^{\prime}(A)$ and $q_I$ is the quotient map of $A$ onto $A/I$.
   New proofs of the equivalence of the property (F) of Tomiyama for
   $A_1\otimes_{\mathrm{min}} A_2$ with certain other properties are presented.

\end{abstract}

\maketitle
% -----------------------------------------------------------
-----------------------------------
\section{Introduction} \label{S:I}

The relationship between the representations and the kernels of representations of the minimal tensor product of two
$C^*$-algebras and those of the factors has been a substantial topic of study since the inception of this method of
building new algebras, as one can see for instance from \cite{W}, \cite{T}, or \cite{G}. Starting with closed two
sided ideals $I_1, I_2$ of the $C^*$-algebras $A_1, A_2$, respectively, one looks at the ideals $I :=
I_1\otimes_{\mathrm{min}} A_2 + A_1\otimes_{\mathrm{min}} I_2$ of $A_1\otimes_{\mathrm{min}} A_2$ and
$\mathrm{ker}(q_{I_1}\otimes q_{I_2})$ which obviously contains $I$. In particular, it is interesting to see what
happens when $I_1$ and $I_2$ are primitive or prime ideals. It is also of interest to know if the map between the
pairs of ideals and the ideals of the tensor product has any continuity properties.

When the primitive ideal spaces are given their hull-kernel topologies it was proved in \cite{W} that the map
$(I_1,I_2)\to \mathrm{ker}(q_{I_1}\otimes q_{I_2})$ is a homeomorphism onto its image. Similar results, under the
assumption that the tensor product enjoys the property (F) of Tomiyama,  were obtained in \cite{H} and \cite{BK} for
the spaces of the kernels of factorial representations and the spaces of prime ideals. We consider here the spaces
of all the proper closed two-sided ideals with their $\tau_w$ topologies, as defined in \cite{A}, and show in
Theorem \ref{T:home} that the same map is a homeomorphism. This establishes, in the most general situation, that the
maps between the spaces of kernels of factorial representations and between the spaces of prime ideals, with their
hull-kernel topologies, are also homeomorphisms.

In section \ref{S:F} we give new proofs, possibly simpler, to the equivalence of the property (F) with various other
properties of the minimal tensor product. For instance, we show that property (F) is equivalent to having every
closed two sided ideal of the minimal tensor product equal to the closed linear span of the elementary tensors
contained in the ideal, a fact first proved in \cite[Proposition 2.16]{BK}. In \cite[Proposition 4.5]{Al} it is
proved that every non-zero closed two sided ideal of any minimal tensor product of two $C^*$-algebras must contain a
non trivial elementary tensor and this fact plays an important role in our proofs, via \cite[Lemma 2.12]{BK}.

By an ideal of a Banach algebra we shall always mean a closed two sided ideal. The set of all the ideals of the
Banach algebra $A$ is denoted by $\mathrm{Id}(A)$ and $\mathrm{Id}^{\prime}(A) := \mathrm{Id}(A)\setminus \{A\}$.
For an ideal $I$ of a Banach algebra $A$ we denote by $q_I$ the quotient homomorphism of $A$ onto $A/I$. Given a
$C^*$-algebra $A$, the set of its primitive ideals, the set of the kernels of its factorial representations, and the
set of its prime ideals are denoted $\Pm(A)$, $\Fa(A)$, and $\Pme(A)$, respectively. Their topology will everywhere
be the hull-kernel topology.

Given a topological space $X$, we shall endow the space of all its closed subsets, $\F(X)$, with the topology for
which a subbase is the collection of all the families $\{F\in \F(X)\mid F\cap O\neq \mset\}$, when $O$ runs through
all the open subsets of $X$.

If $A$ is a $C^*$-algebra the $\tau_w$ topology of $\mathrm{Id}(A)$ is defined by transporting the topology of
$\F(\Pm(A))$ to $\mathrm{Id}(A)$ via the well known correspondence between the closed subsets of $\Pm(A)$ and the
ideals of $A$. Thus a subbase for this topology is given by all the sets $\{I\in \mathrm{Id}(A)\mid J\nsubseteq I\}$
where $J$ is an ideal of $A$, see \cite[p.525]{A}.

$A_1$, $A_2$ being $C^*$-algebras, $A_1\odot A_2$ stands for their algebraic tensor product and $A_1\otimes A_2$ for
their minimal tensor product which will be the only $C^*$-tensor product discussed here. If $B_1$, $B_2$ are also
$C^*$-algebras and $\pi_1 : A_1\to B_1$, $\pi_2 : A_2\to B_2$ are homomorphisms then $\pi_1\otimes \pi_2 ;
A_1\otimes A_2\to B_1\otimes B_2$ is the unique homomorphism determined by $(\pi_1\otimes \pi_2)(a_1\otimes a_2) =
\pi_1(a_1)\otimes \pi_2(a_2)$ for every $a_1\in A_1$, $a_2\in A_2$. The map $\Phi : \mathrm{Id}(A_1)\times
\mathrm{Id}(A_2) \to \mathrm{Id}(A_1\otimes A_2)$ is defined by $\Phi(I_1,I_2) := \mathrm{ker}(q_{I_1}\otimes
q_{I_2})$.

Very useful for the study of the structure of $A_1\otimes A_2$ are the slice maps defined by Tomiyama in \cite{T}.
For a bounded linear functional $f$ on $A_2$, the bounded linear left slice map $L_f : A_1\otimes A_2\to A_1$ is
defined by requiring $L_f(a_1\otimes a_2) := f(a_2)a_1$ for each $a_1\in A_1$, $a_2\in A_2$.

Following \cite{T} one says that $A_1\otimes A_2$ has the property (F) if the set of all product states $f_1\otimes
f_2$, $f_1$ and $f_2$ being pure states of $A_1$ and $A_2$, respectively, separates the ideals of $A_1\otimes A_2$.
Equivalently, by \cite[Theorem 5]{T}, for every ideal $I_1$ of $A_1$ and $I_2$ of $A_2$ one has $\Phi(I_1,I_2) =
I_1\otimes A_2 + A_1\otimes I_2$. It is this property of the ideals that we shall use everywhere in the paper when
we refer to the property (F).

The author is very grateful to Professor R. J. Archbold for pointing out a mistake in an earlier version of this
paper at the same time with providing the way to correct it and for suggesting many other improvements.

\section{The space of ideals of $A_1\otimes A_2$} \label{S:Id}

\begin{lem} \label{L:hyper}

   Let $X_1, X_2$ be topological spaces. The map $\varphi : \F(X_1)\times \F(X_2)\to \F(X_1\times X_2)$ given by
   $\varphi(F_1,F_2) := F_1\times F_2$ is continuous.

\end{lem}

\begin{proof}

   Let $U$ be an open subset of $X_1\times X_2$, $U = \cup (U^1_{\iota}\times
   U^2_{\iota})$ with $U^j_{\iota}$ open subsets of $X_j$. Then
    \begin{multline*}
       \{(F_1,F_2)\in \F(X_1)\times \F(X_2) \mid \varphi(F_1,F_2)\cap U\neq \mset\} =\\
                                                         \cup (\{F_1\in \F(X_1) \mid F_1\cap
       U^1_{\iota} \neq \mset\}\times \{F_2\in \F(X_2) \mid F_2\cap U^2_{\iota} \neq \mset\}).
    \end{multline*}
   This is an open subset of $\F(X_1)\times \F(X_2)$ and the continuity of $\varphi$ is established.

\end{proof}

From here on we let $A_1, A_2$ be $C^*$-algebras and $\Phi$ will be as defined in the introduction. Then $\Phi$ maps
$\Pm(A_1)\times \Pm(A_2)$ into $\Pm(A_1\otimes A_2)$ by \cite{W}, $\Phi(\Fa(A_1)\times \Fa(A_2))\subset
\Fa(A_1\otimes A_2)$ by \cite[p. 6]{G}, and $\Phi(\Pme(A_1)\times \Pme(A_2))\subset \Pme(A_1\otimes A_2)$ by
\cite[Lemma 2.13(v)]{BK}.

\begin{lem} \label{L:intersect}

   Let $\mathcal{M}_k$ be a family of ideals of $A_k$ and
    \[
     I_k := \cap_{J\in \mathcal{M}_k} J
    \]
   for $k = 1,2$. Then $\Phi(I_1,I_2) = \cap \{\Phi(J_1,J_2) \mid (J_1,J_2)\in \mathcal{M}_1\times \mathcal{M}_2\}$.

\end{lem}

\begin{proof}

   For $J_1\in \mathcal{M}_1$ and $J_2\in \mathcal {M}_2$ there exists a natural homomorphism $\theta_{J_1,J_2}$ of
   $(A_1/I_1)\otimes (A_2/I_2)$ onto $(A_1/J_1)\otimes (A_2/J_2)$; one defines it in the obvious way on
   $(A_1/I_1)\odot (A_2/I_2)$ and then one extends it by continuity. This homomorphism satisfies
   $\theta_{J_1,J_2}\circ (q_{I_1}\otimes q_{I_2}) = q_{J_1}\otimes q_{J_2}$ and from this we infer $\Phi(I_1,I_2)\subset
   \Phi(J_1,J_2)$. Define a seminorm on $(A_1/I_1)\otimes (A_2/I_2)$ by
    \[
     N(x) := sup\{\|\theta_{J_1,J_2}(x)\| \mid (J_1,J_2)\in \mathcal{M}_1\times \mathcal{M}_2\}.
    \]
   Then, if $x_1\in A_1/I_1$ and $x_2\in A_2/I_2$ with $x_1\otimes x_2\neq 0$ we have $N(x_1\otimes x_2) > 0$. By
   \cite[Lemma 2.12(i)]{BK} we have $N(x)\geq \|x\|$ for $x\in (A_1/I_1)\otimes (A_2/I_2)$. Hence, if for $a\in
   A_1\otimes A_2$ we have $(q_{J_1}\otimes q_{J_2})(a) = 0$ for each $(J_1,J_2)\in \mathcal{M}_1\times
   \mathcal{M}_2$ then $a\in \Phi(I_1,I_2)$ and we have proved $\Phi(I_1,I_2)\supset \cap \{\Phi(J_1,J_2) \mid
   (J_1,J_2)\in \mathcal{M}_1\times \mathcal{M}_2\}$.

\end{proof}

\begin{cor} \label{C:hulls}

   Let $I_1, I_2$ be ideals in $A_1, A_2$ respectively. Then $\Phi(\mathrm{hull} I_1\times \mathrm{hull} I_2)$ is dense in
   $\mathrm{hull} \Phi(I_1,I_2)$.

\end{cor}

\begin{proof}

   $I_k = \cap_{J\in \mathrm{hull} I_k} J$ hence, by Lemma \ref{L:intersect},
    \[
     \Phi(I_1,I_2) = \cap \{\Phi(J_1,J_2) \mid (J_1,J_2)\in \mathrm{hull} I_1\times \mathrm{hull} I_2\}
    \]
   and from this the conclusion follows.

\end{proof}

Let us denote by $\Phi^{\prime}$ the restriction of $\Phi$ to $\Pm(A_1)\times \Pm(A_2)$. Recall that $\Phi^{\prime}$
is a homeomorphism of $\Pm(A_1)\times \Pm(A_2)$ onto its image which is dense in $\Pm(A_1\otimes A_2)$, see
\cite[Lemme 16]{W}.

\begin{rem} \label{R:Phi}

   Corollary \ref{C:hulls} tells us that for $(I_1,I_2)\in \mathrm{Id}(A_1)\times \mathrm{Id}(A_2)$ we have
   $\Phi(I_1,I_2) = \cap \{\Phi^{\prime}(P_1,P_2) \mid (P_1,P_2)\in \mathrm{hull} I_1\times \mathrm{hull} I_2\}$.

\end{rem}

Now we are ready to prove the continuity of the map $\Phi$. This was proved in \cite[Lemma 1.5]{K} under the
assumption that $A_1\otimes A_2$ has property $(F)$.

\begin{lem} \label{L:cont}

   The map $\Phi$ is continuous.

\end{lem}

\begin{proof}

   Let $U$ be an open subset of $\mathrm{Prim}(A_1\otimes A_2)$. Then, by Corollary \ref{C:hulls},
    \begin{multline*}
     \{(I_1,I_2)\in \mathrm{Id}(A_1)\times \mathrm{Id}(I_2) \mid \mathrm{hull} \Phi(I_1,I_2)\cap U \neq \mset\} =\\
           \{(I_1,I_2)\in \mathrm{Id}(A_1)\times \mathrm{Id}(A_2) \mid \Phi(\mathrm{hull} I_1\times
            \mathrm{hull} I_2)\cap U \neq \mset\} =\\
           \{(I_1,I_2)\in \mathrm{Id}(A_1)\times \mathrm{Id}(A_2) \mid (\mathrm{hull} I_1\times \mathrm{hull}
           I_2)\cap \Phi^{\prime-1}(U) \neq \mset\}
    \end{multline*}
   and the latter set is open in $\mathrm{Id}(A_1)\times \mathrm{Id}(A_2)$ by the continuity of $\Phi^{\prime}$ and
   Lemma \ref{L:hyper}.

\end{proof}

   It has been pointed out to the author by Professor R. J. Archbold that Lemma \ref{L:cont} yields an alternative proof
   of Lemma 1.4 in \cite{ASKS}
   that asserts the continuity of the canonical map from $\mathrm{Id}(A_1)\times \mathrm{Id}(A_2)$ to the space of
   ideals of $A_1\otimes_h A_2$, the Haagerup tensor product of the $C^*$-algebras $A_1$ and $A_2$. It is shown in
   \cite[Lemma 1.1]{ASKS} that this map is obtained by following the map $\Phi$ with the map $K\to K\cap
   (A_1\otimes_h A_2)$ from $\mathrm{Id}(A_1\otimes A_2)$ to $\mathrm{Id}(A_1\otimes_h A_2)$. The continuity of the
   latter map is established on p. 5 of \cite{ASKS} and the conclusion of \cite[Lemma 1.4]{ASKS} follows.

\begin{thm} \label{T:home}

   The map $\Phi$ is a homeomorphism of $\mathrm{Id}^{\prime}(A_1)\times \mathrm{Id}^{\prime}(A_2)$ onto its image which is
   dense in $\mathrm{Id}^{\prime}(A_1\otimes A_2)$.

\end{thm}

\begin{proof}

   There is a map $\Psi$ from $\mathrm{Id}(A_1\otimes A_2)$ to $\mathrm{Id}(A_1)\times \mathrm{Id}(A_2)$ defined as
   follows: for $I\in \mathrm{Id}(A_1\otimes A_2)$ set
    \[
     I_{A_1} := \{a_1\in A_1 \mid a_1\otimes A_2\subset I\}, \quad I_{A_2} := \{a_2\in A_2 \mid A_1\otimes a_2\subset I\}
    \]
   and $\Psi(I) := (I_{A_1},I_{A_2})$, see \cite[Lemma 2.13]{BK} for instance. Then $\Psi\circ \Phi$ restricted to
   $\mathrm{Id}^{\prime}(A_1)\times \mathrm{Id}^{\prime}(A_2)$ is $\mathrm{id}$, the identity map
   of $\mathrm{Id}^{\prime}(A_1)\times \mathrm{Id}^{\prime}(A_2)$. Indeed, let $(I_1,I_2)\in
   \mathrm{Id}^{\prime}(A_1)\times \mathrm{Id}^{\prime}(A_2)$ and set $I := \Phi(I_1,I_2)$. Clearly $I_1\subset I_{A_1}$
   and $I_2\subset I_{A_2}$. Let now $a\in I_{A_1}$ and pick $b\in A_2\setminus I_2$. Then $a\otimes b\in I$
   hence by \cite[Lemma 2.12(iii)]{BK},
   $a\otimes b\in I_1\otimes A_2 + A_1\otimes I_2$. Choose a bounded linear functional $f$ on $A_2$ that vanishes on $I_2$
   and satisfies $f(b) = 1$. By using the left slice map $L_f$ we get $a = L_f(a\otimes b)\in I_1$. We obtained $I_1 =
   I_{A_1}$ and $I_2 = I_{A_2}$ is proved similarly, thus the claim is established. In particular, it follows that the
   restriction of $\Phi$ to the space considered is one to one.
   It is well known that $\Psi$ is
   continuous but we include a proof for completeness since we could find in print only proofs of the continuity of
   some restrictions of $\Psi$. Let $V_n$ be open subsets of $\Pm(A_n)$ and $K_n$ the corresponding ideals of $A_n$,
   and denote $\mathcal{U}_n := \{I_n\in \mathrm{Id}(A_n) \mid \mathrm{hull} I_n\cap V_n \neq \mset\}$, $n = 1,2$.
   Then
    \begin{multline*}
     \Psi^{-1}(\mathcal{U}_1\times \mathcal{U}_2) =
        \{I\in \mathrm{Id}(A_1\otimes A_2) \mid (\mathrm{hull} I_{A_n})\cap V_n \neq \mset, n = 1,2\} =\\
        \{I\in \mathrm{Id}(A_1\otimes A_2) \mid I_{A_n}\nsupseteq K_n, n = 1,2\} =\\
        \{I\in \mathrm{Id}(A_1\otimes A_2) \mid I\nsupseteq K_1\otimes A_2\}\cap
        \{I\in \mathrm{Id}(A_1\otimes A_2) \mid I\nsupseteq A_1\otimes K_2\}
    \end{multline*}
   and the latter is an open subset of $\mathrm{Id}(A_1\otimes A_2)$.
   From $\Psi\circ \Phi = id$ and the continuity of $\Psi$ it follows that $\Phi$ maps open subsets of
   $\mathrm{Id}^{\prime}(A_1)\times \mathrm{Id}^{\prime}(A_2)$ onto relatively
   open sets and the proof that $\Phi$ is a homeomorphism onto its image is complete.

   Let now $\mathcal{W}_r$, $1\leq r\leq m$ be open subsets of $\Pm(A_1\otimes A_2)$. There are
   $\{P^1_r\}^m_1\subset \Pm(A_1)$ and $\{P^2_r\}^m_1\subset \Pm(A_2)$ such that $\Phi^{\prime}(P^1_r,P^2_r)\in
   \mathcal{W}_r$, $1\leq r\leq m$. If we set $I_1 := \cap^m_{r=1} P^1_r$ and $I_2 := \cap^m_{r=1} P^2_r$ then
   \[
    \Phi(I_1,I_2)\in \{I\in \mathrm{Id}(A_1\otimes A_2) \mid \mathrm{hull} I\cap \mathcal{W}_r\neq \mset, \quad
   1\leq r\leq m\}
   \]
   by Remark \ref{R:Phi} and this shows that the image of $\Phi$ is dense in $\mathrm{Id}^{\prime}(A_1\otimes A_2)$.

\end{proof}

\begin{cor} \label{C:Fac}

   The restriction of $\Phi$ to $\Fa(A_1)\times \Fa(A_2)$ is a homeomorphism onto a dense subset of $\Fa(A_1\otimes
   A_2)$.

\end{cor}

\begin{cor} \label{C:Pr}

   The restriction of $\Phi$ to $\Pme(A_1)\times \Pme(A_2)$ is a homeomorphism onto a dense subset of
   $\Pme(A_1\otimes A_2)$.

\end{cor}

The density of the images in the above corollaries follows from the density of $\Phi^{\prime}(\Pm(A_1)\times
\Pm(A_2))$ in $\Pm(A_1\times A_2)$.

Under the assumption that $A_1\otimes A_2$ has property $(F)$, Corollary \ref{C:Fac} appears as Lemma 2.5 in
\cite{H} and Corollary \ref{C:Pr} is part of Proposition 2.16 of \cite{BK}. Property $(F)$ ensures that the
restrictions of $\Phi$ are onto maps, see also Proposition \ref{P:F} below.

\section{The property $(F)$} \label{S:F}

We continue to use the same notation as in the previous section: $A_1$ and $A_2$ are $C^*$-algebras, $\Phi :
\mathrm{Id}(A_1)\times \mathrm{Id}(A_2)\to \mathrm{Id}(A_1\otimes A_2)$ is the map defined in the introduction and
$\Psi : \mathrm{Id}(A_1\otimes A_2)\to \mathrm{Id}(A_1)\times \mathrm{Id}(A_2)$ is the map defined in the proof of
Theorem \ref{T:home}. Then $\Psi(\Fa(A_1\otimes A_2))\subset \Fa(A_1)\times \Fa(A_2)$ by \cite[Proposition 1]{G} and
$\Psi(\Pme(A_1\otimes A_2))\subset \Pme(A_1)\times \Pme(A_2)$ by \cite[Lemma 2.13(i)]{BK}. It is not known if $\Psi$
maps $\Pm(A_1\otimes A_2)$ into $\Pm(A_1)\times \Pm(A_2)$, even if $A_1\otimes A_2$ has the property $(F)$. We shall
use also the map $\Delta : \mathrm{Id}(A_1)\times \mathrm{Id}(A_2)\to \mathrm{Id}(A_1\otimes A_2)$ defined by
$\Delta(I_1,I_2) := I_1\otimes A_2 + A_1\otimes I_2$. By using slice maps one easily sees that
$\Delta(\mathrm{Id}^{\prime}(A_1)\times \mathrm{Id}^{\prime}(A_2))\subset \mathrm{Id}^{\prime}(A_1\otimes A_2)$.
From \cite[Lemma 2.12(iv)]{BK} we infer that the map $\Delta$ is one to one on $\mathrm{Id}^{\prime}(A_1)\times
\mathrm{Id}^{\prime}(A_2)$.

In the the result below we collect several conditions that are equivalent to the property $(F)$. All of them appear,
possibly in a slightly different form, either in \cite[Proposition 2.2]{H} or in \cite[Proposition 2.16]{BK} and we
only provide new proofs to various implications.

\begin{prop} \label{P:F}

The following properties of $A_1\otimes A_2$ are equivalent:
 \begin{itemize}
    \item[(1)] $A_1\otimes A_2$ has property $(F)$.
    \item[(2)] For every prime ideal $I$ of $A_1\otimes A_2$ we have $I = \Phi(\Psi(I))$.
    \item[(3)] $\Phi(\Pme(A_1)\times \Pme(A_2)) = \Pme(A_1\otimes A_2)$.
    \item[(3a)] $\Delta(\Pme(A_1)\times \Pme(A_2)) = \Pme(A_1\otimes A_2)$.
    \item[(4)] For every $I\in \Fa(A_1\otimes A_2)$ we have $I = \Phi(\Psi(I))$.
    \item[(5)] $\Phi(\Fa(A_1)\times \Fa(A_2)) = \Fa(A_1\otimes A_2)$.
    \item[(5a)] $\Delta(\Fa(A_1)\times \Fa(A_2)) = \Fa(A_1\otimes A_2)$.
    \item[(6)] Every ideal $I$ of $A_1\otimes A_2$ is the closure of the linear span of the elementary tensors contained
    in $I$.
 \end{itemize}

\end{prop}

\begin{proof}

   (1) $\Rightarrow$  (2). Let $I\in \Pme(A_1\otimes A_2)$ and set $(I_1,I_2) = \Psi(I)$. Then $I\subset
   \Phi(I_1,I_2) = I_1\otimes A_2 + A_1\otimes I_2$ by \cite[Lemma 2.13(iv)]{BK} and our assumption. From the
   definition of $\Psi$ follows
   $I_1\otimes A_2 + A_1\otimes I_2\subset I$ so we obtained $I = \Phi(\Psi(I))$.

   (2) $\Rightarrow$ (3). $\Phi(\Pme(A_1)\times \Pme(A_2))\subset \Pme(A_1\otimes A_2)$ by \cite[Lemma 2.13(v)]{BK}
   and $\Psi(\Pme(A_1\otimes A_2))\subset \Pme(A_1)\times \Pme(A_2)$ by \cite[Lemma 2.13(i)]{BK} so (3) follows
   immediately from (2).

   (3) $\Rightarrow$ (2). Let $I\in \Pme(A_1\otimes A_2)$. Then $I = \Phi(I_1,I_2)$ with $(I_1,I_2)\in \Pme(A_1)\times
   \Pme(A_2)$. As shown in the proof of Theorem \ref{T:home}, $\Psi\circ \Phi$ is the identity map of
   $\mathrm{Id}^{\prime}(A_1)\times \mathrm{Id}^{\prime}(A_2)$ hence $\Psi(I) = (I_1,I_2)$ and (2) obtains.

   (2) $\Rightarrow$ (4) is trivial and (4) $\Rightarrow$ (1) is part of \cite[Theorem 5]{T}. Now, clearly (1) and
   (3) imply (3a).

   (3a) $\Rightarrow$ (2). Let $I$ be a prime ideal of $A_1\otimes A_2$ and set $\Psi(I) = (I_1,I_2)\in
   \Pme(A_1)\times \Pme(A_2)$, by \cite[Lemma 2.13(i)]{BK}. Then $J := \Phi(I_1,I_2)\in \Pme(A_1\otimes A_2)$ by
   \cite[Lemma 2.13(v)]{BK} and there exist $J_1\in \Pme(A_1), \; J_2\in \Pme(A_2)$ such that $\Delta(J_1,J_2) = J =
   \Phi(I_1,I_2)$. Then, again by \cite[Lemma 2.13(v)]{BK}, $J_1 = I_1, \; J_2 = I_2$. Now, $\Delta(I_1,I_2)\subset
   I\subset \Phi(I_1,I_2) = \Delta(I_1,I_2)$, the second inclusion following from \cite[Lemma 2.13(iv)]{BK}. Thus we
   got $I = \Phi(I_1,I_2) = \Phi(\Psi(I))$.

   (4) $\Leftrightarrow$ (5) is obtained as (2) $\Leftrightarrow$ (3) above by using \cite[p. 6 and Proposition 1]{G}.
   (1), (5) $\Rightarrow$ (5a) is obvious. (5a) $\Rightarrow$ (4) can be deduced, with obvious changes, as
   (3a) $\Rightarrow$ (2).

   (1), (3a) $\Rightarrow$ (6). For a $C^*$-algebra $A$ and an ideal $K$ of $A$ we shall denote by $\Pme(K)$ the open
   subset of $\Pme(A)$ consisting of all the prime ideals that do not contain $K$.

   Let $I$ be an ideal of $A_1\otimes A_2$ and pick $P\in \Pme(I)$. We claim that
   there exist ideals $J_1$ of $A_1$ and $J_2$ of $A_2$ such that $P\in \Pme(J_1\otimes J_2)$ and $\Pme(J_1\otimes
   J_2)\subset \Pme(I)$. Once this claim is established then we can assert that the ideal $I^{\prime}$ generated by
   all the elementary tensors contained in $I$ satisfies $\Pme(I^{\prime}) = \Pme(I)$ hence $I^{\prime} = I$ and we
   are done.

   From the assumption on $A_1\otimes A_2$ and Corollary \ref{C:Pr} it follows that $\Phi = \Delta$ and there are open
   subsets $U_m$ of $\Pme(A_m)$ such that $P\in \Delta(U_1\times U_2)$ and $\Delta(U_1\times U_2)\subset \Pme(I)$.
   Denote by $J_m$ the ideal of $A_m$ that satisfies $\Pme(J_m) = U_m$. To substantiate the above claim it remains
   to show that $\Pme(J_1\otimes J_2) = \Delta(U_1\times U_2)$. Now, if $Q\in \Pme(J_1\otimes J_2)$ that is, $Q\in
   \Pme(A_1\otimes A_2)$ but $Q\nsupseteq J_1\otimes J_2$, and $Q = \Delta(Q_1,Q_2) = Q_1\otimes A_2 + A_1\otimes Q_2$
   with $(Q_1,Q_2)\in \Pme(A_1)\times \Pme(A_2)$ then $Q_1\nsupseteq J_1$ and $Q_2\nsupseteq J_2$. Hence
   $\Pme(J_1\otimes J_2)\subseteq \Delta(U_1\times U_2)$. On the other hand, we are going to show that if $Q\in
   \Pme(A_1\otimes A_2)$ contains $J_1\otimes J_2$ and $Q_1, Q_2$ are the prime ideals of $A_1, A_2$, respectively,
   that satisfy $Q = \Delta(Q_1,Q_2)$ then $J_1\subset Q_1$ or $J_2\subset Q_2$. So suppose $J_1\otimes J_2\subset
   Q_1\otimes A_2 + A_1\otimes Q_2$ and assume that there exists $a_2\in J_2\setminus Q_2$. Choose a bounded linear
   functional $f$ on $A_2$ that vanishes on $Q_2$ but $f(a_2) = 1$. Then if $a_1\in J_1$, we have $a_1 =
   L_f(a_1\otimes a_2)\subset Q_1$ hence $J_1\subset Q_1$.

   (6) $\Rightarrow$ (1). Let $I_m$ be an ideal of $A_m$. By our assumption on the ideals of $A_1\otimes A_2$ and
   \cite[Lemma 2.12(iii)]{BK} we have $\Phi(I_1,I_2) = \Delta(I_1,I_2)$.

\end{proof}

\bibliographystyle{amsplain}
\bibliography{}

\end{document}